\documentclass[12pt]{article}
\usepackage{amssymb, amsthm}
\usepackage{amsmath}

\newtheorem{Rem}{Remark}
\newtheorem{sats}{Theorem}
\newtheorem{prop}{Proposition}
\newtheorem{lem}{Lemma}

\newcommand{\banm}{\begin{anm}}
\newcommand{\eanm}{\end{anm}}

\newcommand{\loc}{{\rm loc}}

\newcommand{\const}{{\rm const}}

\begin{document}

\title
{Oblique derivative problem \\
for  non-divergence parabolic equations\\
with discontinuous in time coefficients\\ in a wedge}
\author
{Vladimir Kozlov\footnote{Department of Mathematics, University of Link\"oping, SE-581 83
Link\"oping, Sweden} \ and Alexander Nazarov\footnote{St.-Petersburg Department of Steklov
Mathematical Institute, Fontanka, 27, St.-Petersburg, 191023, Russia, and
St.-Petersburg State University, Universitetskii pr. 28, St.-Petersburg, 198504, Russia}
}

\date{}
\maketitle
\begin{abstract}

\noindent We consider an oblique derivative problem in a wedge for non-di\-ver\-gence parabolic
equations with discontinuous in $t$ coefficients. We obtain
weighted coercive estimates of solutions in anisotropic Sobolev spaces.

\end{abstract}

\section{Introduction}


Consider the parabolic differential operator
\begin{equation}\label{Jan1}
{\cal L} \equiv\partial_t-a^{ij}(t)D_iD_j,
\end{equation}
where $x\in\mathbb R^n$, $t\in\mathbb R$ and the convention regarding summation from $1$ to $n$
with respect to repeated indices is adopted. Here and elsewhere $D_i$
denotes the operator of differentiation with respect to $x_i$, $i=1,\ldots,n$, $D=(D_1,\ldots,D_n)$, and
$\partial_t$ denotes differentiation with respect to $t$.
We assume that $a^{ij}$ are measurable real
valued functions of $t$ satisfying $a^{ij}=a^{ji}$ and
\begin{equation}\label{Jan2}
\nu|\xi|^2\le a^{ij}\xi_i\xi_j\le \nu^{-1}|\xi|^2, \qquad
\xi\in{\mathbb R}^n, \quad \nu=\const>0.
\end{equation}

In \cite{Kr2}, \cite{Kr} it was shown by Krylov that for coercive
estimates of $\partial_tu$ and $D(Du)$ one needs no smoothness
assumptions on coefficients $a^{ij}$ with respect to $t$. The only
assumption which is needed is estimate (\ref{Jan2}). Solvability
results in the whole space $\mathbb R^n\times\mathbb R$ for
equation (\ref{Jan1}) in $L_{p,q}$ spaces were proved in \cite{Kr2};
solvability of the Dirichlet problem in the half-space $\mathbb
R^n_+\times\mathbb R$ in weighted $L_{p,q}$ spaces was established
by Krylov \cite{Kr} for particular range of weights and by the
authors \cite{KN} for the whole range of weights. Similar estimates for the oblique derivative problem in a half-space
are obtained in \cite{KN2}, see also \cite{DKZ}.
\medskip

This paper addresses solvability results for the boundary value problems to (\ref{Jan1}) in the wedge. Namely, let $n\geq m\geq 2$ and
let $K$ be a  cone in $\mathbb R^m$. We assume that  the boundary of $\omega=K\cap \mathbb S^{m-1}$ is of class $C^{1,1}$,
where $\mathbb S^{m-1}$ is the unit sphere in $\mathbb R^m$. We put ${\cal K}=K\times\mathbb R^{n-m}$. We underline that the case $m=n$ where
${\cal K}=K$ is not excluded. In what follows we use the representation $x=(x',x'')$, where $x'\in\mathbb R^m$ and $x''\in\mathbb R^{n-m}$.

First, we recall the important notion of {\it critical exponents} for the operator ${\cal L}$ in
the wedge ${\cal K}$. They were introduced in \cite{KN1}, where the following
Dirichlet problem was considered:
\begin{equation}\label{Jan1a}
{\cal L}u=f\quad\mbox{in}\quad {\cal K}\times \mathbb R;\qquad
u=0\quad\mbox{on}\quad \partial {\cal K}\times \mathbb R.
\end{equation}
To define them we need the space ${\cal V}(Q_R^{\cal K}(t_0))$ of functions $u$ with finite norm
\begin{equation*}
\|u\|_{{\cal V}(Q_R^{\cal K}(t_0))}=\sup_{\tau\in\,(t_0-R^2,t_0]}\|u(\cdot,\tau)\|_{L_2(B_R\cap{\cal K})}
+\|Du\|_{L_2(Q_R^{\cal K}(t_0)))},
\end{equation*}
where
$$
B_R=\{x\,:\ |x|<R\};\qquad Q_R^{\cal K}(t_0)=(B_R\cap {\cal K})\times(t_0-R^2,t_0].
$$
We write $u\in {\cal V}_\loc(Q_R^{\cal K}(t_0))$ if
$u\in {\cal V}(Q_{R'}^{\cal K}(t_0))$ for all $R'\in (0,R)$.\medskip

The first critical exponent is defined as the supremum of all real $\lambda$
such that
\begin{equation}\label{0.4}
|u(x;t)|\leq C(\lambda,\kappa) \Big (\frac{|x'|}{R}\Big )^\lambda\sup_{
Q_{\kappa R}^{\cal K}(t_0)}|u|\quad\mbox{for}\quad (x;t)\in Q_{R/2}^{\cal K}(t_0)
\end{equation}
for a certain $\kappa\in (1/2,1)$ independent of $t_0$, $R$ and $u$.
This inequality must be satisfied for all  $t_0\in\mathbb R$, $R>0$ and
for all  $u\in {\cal V}_\loc(Q_{R}^{\cal K}(t_0))$  subject to
\begin{equation*}\label{0.5}
{\cal L}u=0\quad\mbox{in}\quad Q_R^{\cal K}(t_0);\qquad
u\big|_{x\in\partial{\cal K}}=0.
\end{equation*}
 We denote this critical exponent by $\lambda_c^+\equiv\lambda_c^+({\cal K},{\cal L})$.
It is shown in \cite{KN1} that definition of $\lambda_c^+$ does not depend on $\kappa$.
Since $\lambda=0$  satisfies (\ref{0.4}) we conclude that $\lambda_c^+\geq0$.

The second critical exponent is defined as $\lambda_c^-\equiv\lambda_c^-({\cal K},{\cal L}):=\lambda_c^+({\cal K},\widehat{{\cal L}})$, where
\begin{equation*}\label{0.51}
\widehat{{\cal L}}=\partial_t-a^{ij}(-t)D_iD_j.
\end{equation*}

Several properties of these critical exponents are established in \cite[Theorem 2.4]{KN1}. In particular,
\begin{enumerate}
 \item[(i)] $\lambda_c^+$ and $\lambda_c^-$ are strictly positive;
 \item[(ii)] they depend monotonically on ${\cal K}$. This means that if ${\cal K}_1\subset{\cal K}_2$ then
 $\lambda_c^+({\cal K}_1,{\cal L})\ge\lambda_c^+({\cal K}_2,{\cal L})$ and
$\lambda_c^-({\cal K}_1,{\cal L})\ge\lambda_c^-({\cal K}_2,{\cal L})$;
 \item[(iii)] if $a^{ij}=\delta_{ij}$ then
$$
\lambda_c^{\pm}({\cal K},{\cal L})=
\lambda_{\cal D}\equiv-\frac{m-2}{2}+\sqrt{\Lambda_{\cal D}+\frac{(m-2)^2}{4}},
$$
 where $\Lambda_{\cal D}$ is the first eigenvalue of the Dirichlet boundary value problem to
the Beltrami-Laplacian in $\omega$;
\item[(iv)] Let ${\cal L}'=\partial_t-\sum_{k,\ell=1}^mA_{k\ell}(t)D_kD_{\ell}$.
Then $\lambda_c^{\pm}({\cal K},{\cal L})=\lambda_c^{\pm}(K,{\cal L}')$.
 \end{enumerate}

In order to formulate the main result of \cite{KN1} we introduce two classes of anisotropic spaces.
For $1<p,q<\infty$ we define $L_{p,q}=L_{p,q}({\cal K}\times \mathbb R)$ and $\tilde{L}_{p,q}=\tilde{L}_{p,q}({\cal K}\times\mathbb R)$
as spaces of functions $f$ with finite norms
$$
\|f\|_{p,q}:=\big\|\|f(\cdot;t)\|_{p,{\cal K}}\big\|_{q,\mathbb R}=
\Big (\int\limits_{\mathbb R}\Big (\int\limits_{\cal K}|f(x;t)|^pdx\Big )^{q/p}dt\Big )^{1/q}
$$
and
$$|\!|\!|f|\!|\!|_{p,q}:=\big\|\|f(x,\cdot)\|_{q,\mathbb R}\big\|_{p,{\cal K}}=
\Big (\int\limits_{\cal K}\Big (\int\limits_{\mathbb R}|f(x;t)|^qdt\Big )^{p/q}dx\Big )^{1/p}
$$
respectively.

\begin{prop}[\cite{KN1}, Theorem 1.1]
 Suppose that
\begin{equation*}
2-\frac mp-\lambda_c^+<\mu<m-\frac mp+\lambda_c^-.
\end{equation*}
Then for any $f$ such that $|x'|^\mu f\in L_{p,q}$ (respectively, $|x'|^\mu f\in \widetilde L_{p,q}$)
there is a solution of the boundary value problem {\rm (\ref{Jan1a})} satisfying the following estimates:
\begin{equation}\label{beztilde}
\||x'|^\mu\partial_t u\|_{p,q}+\||x'|^\mu D(Du)\|_{p,q}+
\||x'|^{\mu-2} u\|_{p,q}\le C\ \||x'|^\mu f\|_{p,q};
\end{equation}
\begin{equation}\label{tilde}
|\!|\!||x'|^\mu\partial_t u|\!|\!|_{p,q}+
|\!|\!||x'|^\mu D(Du)|\!|\!|_{p,q}+ |\!|\!| |x'|^{\mu-2} u|\!|\!|_{p,q}\le C\ |\!|\!||x'|^\mu f|\!|\!|_{p,q}.
\end{equation}
This solution is unique in the space of functions with the
finite norms in the left-hand side of {\rm (\ref{beztilde})} (respectively, of {\rm (\ref{tilde}))}.
\end{prop}

Here we complement this result by considering the problem
\begin{equation}\label{2.6}
Lu=f_0+\mbox{div}\,({\bf f})\qquad\mbox{in}\quad \mathbb {\cal K}\times\mathbb R;\qquad u\big|_{x\in\partial{\cal K}}=0,
\end{equation}
where ${\bf f}=(f_1,\ldots,f_n)$. Denote by $\Gamma^{\cal D}_{\cal K}$ the Green function of problem (\ref{Jan1a})
(this Green function was constructed in \cite{KN1}). One of the results of this paper is the following.

\begin{sats}\label{prop11} Let
\begin{equation}\label{mu}
1-\frac mp-\lambda_c^+<\mu<m-1-\frac mp+\lambda_c^-.
\end{equation}

\noindent {\rm (i)} Suppose that $|x'|^{\mu+1}f_0,\,|x'|^\mu {\bf f}\in
\widetilde{L}_{p,q}({\cal K}\times\mathbb R)$. Then the function
\begin{equation}\label{2.8}
u(x;t)=\int\limits_{-\infty}^t\int\limits_{\cal K}\Big(\Gamma^{\cal D}_{\cal K}(x,y;t,s)f_0(y;s)
-D_y\Gamma^{\cal D}_{\cal K}(x,y;t,s)\cdot{\bf f}(y;s)\Big)\,dyds
\end{equation}
gives a weak solution to problem {\rm (\ref{2.6})} and satisfies the estimate
\begin{equation}\label{2.9}
|\!|\!||x'|^\mu D u|\!|\!|_{p,q}+|\!|\!||x'|^{\mu-1} u|\!|\!|_{p,q}\leq
C(|\!|\!||x'|^{\mu+1}f_0|\!|\!|_{p,q}+|\!|\!||x'|^\mu {\bf f}|\!|\!|_{p,q}).
\end{equation}

\noindent {\rm (ii)} Suppose that $|x'|^{\mu+1}f_0,\,|x'|^\mu {\bf f}\in
L_{p,q}({\cal K}\times\mathbb R)$. Then the function {\rm (\ref{2.8})}
gives a weak solution to problem {\rm (\ref{2.6})} and satisfies the estimate
\begin{equation}\label{2.9a}
\||x'|^\mu D u\|_{p,q}+\||x'|^{\mu-1} u\|_{p,q}\leq C(\||x'|^{\mu+1}
f_0\|_{p,q}+\||x'|^\mu {\bf f}\|_{p,q}).
\end{equation}
\end{sats}

Let us turn to  the oblique derivative problem in ${\cal K}$
\begin{equation}\label{2.6zz}
Lu=f\quad\mbox{in}\quad \mathbb {\cal K}\times\mathbb R;\qquad D_1 u=0\quad\mbox{on}\quad (\partial K\setminus O)\times\mathbb R^{n-m}\times\mathbb R.
\end{equation}

We assume additionally that the cone $K$ is strictly Lipschitz, i.e.
\begin{equation}\label{LUD1}
K=\{x'=(x_1,\hat{x})\,:\, x_1>\phi(\hat{x}),\;\hat{x}\in\mathbb R^{m-1}\},
\end{equation}
where $\phi$ is a Lipschitz function, i.e.
\begin{equation}\label{LUD2}
|\phi(\hat{x})-\phi(\hat{y})|\leq \Lambda |\hat{x}-\hat{y}|\qquad\mbox{for all}
\quad\hat{x},\,\hat{y}\in\mathbb R^{m-1}.
\end{equation}
The main result of this paper is the following.

\begin{sats}\label{Th1s} Let $1<p,q<\infty$ and let $\mu$ be subject to
\begin{equation}\label{mu1}
-\frac mp+(1-\lambda_c^+)_+<\mu<m-\frac mp-(1-\lambda_c^-)_+
\end{equation}
{\rm (}here $a_+=\max \{a,0\}${\rm )}.
\smallskip

{\rm (i)} If $|x'|^\mu f\in\widetilde{L}_{p,q}({\cal K}\times\mathbb R)$ then solution to problem {\rm (\ref{2.6zz})} defined by
{\rm (\ref{2.11ag})}, {\rm (\ref{Green_N})}, satisfies
\begin{equation}\label{TTn1a}
 |\!|\!||x'|^\mu\partial_t u|\!|\!|_{p,q}+|\!|\!||x'|^\mu D(Du)|\!|\!|_{p,q}\le
C\ |\!|\!||x'|^\mu f|\!|\!|_{p,q}.
\end{equation}

{\rm (ii)} If $|x'|^\mu f\in L_{p,q}({\cal K}\times\mathbb R)$ then solution to problem {\rm (\ref{2.6zz})} defined by
{\rm (\ref{2.11ag})}, {\rm (\ref{Green_N})}, satisfies
\begin{equation}\label{TTn1b}
\||x'|^\mu\partial_t u\|_{p,q}+\||x'|^\mu D(Du)\|_{p,q}\le C\ \||x'|^\mu f\|_{p,q}.
\end{equation}
The constant $C$ depends only on $\nu$, $\mu$, $p$, $q$ and ${\cal K}$.
\end{sats}

\begin{Rem}
 Since $\lambda_c^+$ and $\lambda_c^-$ are positive, the intervals for $\mu$ in {\rm (\ref{mu})} and  {\rm (\ref{mu1})} are non-empty even for $m=2$.
\end{Rem}

We use an approach based on the study of the Green functions. In Section \ref{Dir} we collect
(partially known) results on the estimate of the Green function for
equation (\ref{Jan1}) in the whole space and in the wedge subject to the Dirichlet boundary condition. In particular in Sect. \ref{WDir} we prove
Theorem \ref{prop11}.

Section \ref{Neu} is devoted to the estimates of the Green function for the oblique derivative problem
 and to the proof of Theorem \ref{Th1s}.\medskip

Let us recall some notation. In what follows we denote by the same letter the kernel  and the
corresponding integral operator, i.e.
\begin{equation*}
({\cal T}h)(x;t)=\int\limits_{-\infty}^t\int\limits_{\mathbb R^n}
{\cal T}(x,y;t,s)h(y;s)\,dyds.
 \end{equation*}
Here we expand functions $\cal T$ and $h$ by zero to whole space-time if necessary.\medskip

For $x\in{\cal K}$, $d(x)$ is the distance from $x$ to $\partial {\cal K}$, and
$r_x=\frac{d(x)}{|x'|}$. We set
\begin{equation*}
{\cal R}_{x,t}=\frac{|x'|}{|x'|+\sqrt{t}};\qquad
{\cal R}_{y,t}=\frac{|y'|}{|y'|+\sqrt{t}}.
\end{equation*}
We use the letter $C$ to denote
various positive constants. To indicate that $C$ depends on some
parameter $a$, we sometimes write $C(a)$.

\section{Preliminary results}\label{Dir}
\subsection{The estimates in the whole space}\label{Rn}

Denote by $\Gamma$ the Green function of the operator ${\cal L}$
in the whole space:
\begin{equation*}
\Gamma(x,y;t,s)= \frac { \det\big(\int_s^t A(\tau)d\tau\big)^{-\frac
12}} {(4\pi)^{\frac n2}} \exp \bigg(-\frac {\Big(\big(\int_s^t
A(\tau) d\tau\big)^{-1} (x-y),(x-y)\Big)}4\bigg)
\end{equation*}
for $t>s$ and $0$ otherwise.  Here  $A(t)$ is  the matrix $\{
a^{ij}(t)\}_{i,j=1}^n$. The above representation implies, in
particular, the following estimates.

\begin{prop}\label{Pr1} Let $\alpha$ and $\beta$ be two arbitrary multi-indices. Then
\begin{equation}\label{May0}
|D_x^\alpha D_y^\beta \Gamma(x,y;t,s)|\le C\,(t-s)^{-\frac
{n+|\alpha|+|\beta|}2} \,\exp\left(-\frac{\sigma|x-y|^2}{t-s}\right),
\end{equation}
\begin{equation}\label{May0'}
|D_x^\alpha D_y^\beta\partial_s \Gamma(x,y;t,s)|\le C\,(t-s)^{-\frac
{n+|\alpha|+|\beta|+2}2}
\,\exp\left(-\frac{\sigma|x-y|^2}{t-s}\right)
\end{equation}
for $x,y\in\mathbb R^n$ and $s<t$. Here $\sigma$ depends only on the
ellipticity constant $\nu$ and $C$ may depend on $\nu$, $\alpha$ and $\beta$.
\end{prop}

The following statement is a particular case of a general result on boundedness of singular operators
in Lebesgue spaces with Muckenhoupt weights. For $p=q$ it can be extracted from \cite{F}. However,
we could not find corresponding result for anisotropic spaces and give here a direct proof.

\begin{sats}\label{whole} Let $1<p,q<\infty$, and let
\begin{equation}\label{mu2}
-\frac mp<\mu<m-\frac mp.
\end{equation}
Then the integral operator with the kernel
$\frac{|x'|^\mu}{|y'|^\mu}D_xD_x{\Gamma}(x,y;t,s)$
is bounded in $\widetilde{L}_{p,q}(\mathbb R^n\times\mathbb R)$ and  $L_{p,q}(\mathbb R^n\times\mathbb R)$ spaces.
\end{sats}

\begin{proof}
For $\mu=0$ this statement is proved in \cite{Kr} (for the space $L_{p,q}$) and in \cite{KN}
(for the space $\widetilde{L}_{p,q}$). Thus, it is sufficient to prove boundedness of the operator
with kernel
$$
{\cal G}(x,y;t,s)=\Big(\frac{|x'|^\mu}{|y'|^\mu}-1\Big)D_xD_x{\Gamma}(x,y;t,s).
$$
Using (\ref{May0}), (\ref{May0'}) and elementary estimates we obtain
\begin{equation}\label{Apr19}
|{\cal G}(x,y;t,s)|\le C\,\frac {|x'|^{-r}\cdot\Phi\big(\frac {|x'|}{|y'|}\big)}
{(t-s)^{\frac {n+2-r}2}}\,\exp\left(-\frac{\sigma|x-y|^2}{2(t-s)}\right),
\end{equation}
\begin{equation}\label{Apr19'}
|\partial_s {\cal G}(x,y;t,s)|\le C\,\frac {|x'|^{-r}\cdot\Phi\big(\frac {|x'|}{|y'|}\big)}
{(t-s)^{\frac {n+4-r}2}}\,\exp\left(-\frac{\sigma|x-y|^2}{2(t-s)}\right)
\end{equation}
for $x,y\in\mathbb R^n$ and $s<t$. Here $r=\min\{|\mu|, 1\}$ and
$$
\Phi(t)=
\begin{cases}
t^\mu+t, & \mu>1;\\
t^\mu, & 0<\mu\le 1;\\
1, &-1\le\mu<0;\\
t^{\mu+1}+1, & \mu<-1.
\end{cases}
$$
Due to (\ref{Apr19}) and (\ref{mu2}), the kernels ${\cal G}$ satisfy the conditions of
Proposition \ref{L_p} (see Appendix) with $\lambda_1=-r$, $\lambda_2=0$, $\varepsilon_1=\varepsilon_2=0$.
Thus, the operator ${\cal G}$ is bounded in $L_p(\mathbb R^n\times\mathbb R)$ and 
$\widetilde L_{p,\infty}(\mathbb R^n\times\mathbb R)$. Generalized Riesz--Thorin theorem, see, e.g.,
\cite[1.18.7]{Tr}, shows that this operator is bounded in
$\widetilde L_{p,q}(\mathbb R^n\times\mathbb R)$ for any $q\ge p$.

For $q<p$ the same argument provides boundedness of the operator ${\cal G}^*$
in $L_{p'}(\mathbb R^n\times\mathbb R)$ and  $\widetilde L_{p',\infty}(\mathbb R^n\times\mathbb R)$ spaces,
and thus in $\widetilde L_{p',q'}(\mathbb R^n\times\mathbb R)$. Now duality
argument gives the statement of Theorem for the spaces $\widetilde L_{p,q}$.\medskip

To deal with the scale $L_{p,q}$, we take a function $h$ supported in the layer $|s-s^0|\le\delta$ such
that $\int h(y;s)\, ds\equiv 0$. Then
$$
({\cal G}h)(x;t)=\int\limits_{-\infty}^{t}\int\limits_{\mathbb R^n}
\Bigl({\cal G}(x,y;t, s)-{\cal G}(x,y;t,s^0)\Bigr)\, h(y; s)\, dyds.
$$
For $|s- s^0|<\delta$ and $t- s^0>2\delta$, estimate (\ref{Apr19'}) implies
\begin{eqnarray*}
&&\left|{\cal G}(x,y;t, s)-{\cal G}(x,y;t, s^0)\right|
\le\int\limits_{s^0}^s|\partial_\tau {\cal G}_j(x,y;t,\tau)|\,d\tau\\
&&\le C\,\frac {|x'|^{-r}\cdot\Phi\big(\frac {|x'|}{|y'|}\big)}
{(t-s)^{\frac {n+2-r}2}}\, \frac {\delta} {t- s}\,\exp\left(-\frac{\sigma|x-y|^2}{2(t-s)}\right).
\end{eqnarray*}
Thus, our kernel satisfies the assumptions of Proposition \ref{L_p_1} (see Appendix) with $\varkappa=1$, $\lambda_1=-r$,
$\lambda_2=0$, $\varepsilon_1=\varepsilon_2=0$. This gives
$$
\int\limits_{|t- s^0|>2\delta}\Vert ({\cal G}h)(\cdot; t)\Vert_p\, dt\le C\,\Vert h\Vert_{p,1},
$$
where $C$ does not depend on $\delta$ and $ s^0$.

The last estimate is equivalent to the second condition in \cite[Theorem 3.8]{BIN} while the boundedness
of ${\cal G}$ in $L_p(\mathbb R^n\times \mathbb R)$ gives the first condition in this theorem.
Therefore, Theorem 3.8 \cite{BIN} ensures that ${\cal G}$ is bounded in
$L_{p,q}(\mathbb R^n\times \mathbb R)$ for any  $1<q<p<\infty$. For $q>p$ this statement follows by
duality arguments.\end{proof}

\subsection{Coercive estimates for weak solutions to the Dirichlet
problem in the wedge}\label{WDir}

We recall that $\Gamma^{\cal D}_{\cal K}$ stands for the Green function of the operator ${\cal L}$
in the wedge ${\cal K}$ under the Dirichlet boundary condition, and $\lambda_c^{\pm}>0$ are the
critical exponents for ${\cal L}$ in ${\cal K}$.\medskip

The next statement is proved in \cite[Theorem 3.10]{KN1} for $|\alpha'|, \, |\beta'|\leq 2$.
In general case the proof runs almost without changes.

\begin{prop}\label{Green_D} Let $\lambda^+<\lambda_c^+$ and $\lambda^-<\lambda_c^-$.
For $x,y\in {\cal K}$, $t>s$ the following estimates are valid:
\begin{eqnarray}\label{May3}
&&|D_{x}^{\alpha}D_{y}^{\beta} \Gamma^{\cal D}_{\cal K}(x,y;t,s)|\le
C\,\frac {{\cal R}_{x,t-s}^{\lambda^+-|\alpha'|}}{r_x^{(|\alpha'|-2+\varepsilon)_+}}
\ \frac {{\cal R}_{y,t-s}^{\lambda^--|\beta'|}}{r_y^{(|\beta'|-2+\varepsilon)_+}}
\nonumber\\
&&\times{(t-s)^{-\frac{n+|\alpha|+|\beta|}2}} \,\exp
\left(-\frac{\sigma|x-y|^2}{t-s}\right),
\end{eqnarray}
\begin{eqnarray}\label{May4}
&&|D_{x}^{\alpha}D_{y}^{\beta}\partial_s \Gamma^{\cal D}_{\cal K}(x,y;t,s)|\le
C\,\frac {{\cal R}_{x,t-s}^{\lambda^+-|\alpha'|}}{r_x^{(|\alpha'|-2+\varepsilon)_+}}
\ \frac {{\cal R}_{y,t-s}^{\lambda^--|\beta'|-2}}{r_y^{|\beta'|+\varepsilon}}
\nonumber\\
&&\times{(t-s)^{-\frac {n+|\alpha|+|\beta|+2}2}} \,\exp
\left(-\frac{\sigma|x-y|^2}{t-s}\right),
\end{eqnarray}
where $\sigma$ is a positive constant depending on $\nu$,
$\varepsilon$ is an arbitrary small positive number  and $C$ may
depend on $\nu$, $\lambda^{\pm}$, $\alpha$, $\beta$ and $\varepsilon$.
\end{prop}

Now we turn to problem (\ref{2.6}) and to the proof of Theorem \ref{prop11}. \medskip

{\bf Proof of Theorem \ref{prop11}}.  First, function (\ref{2.8}) obviously solves
problem (\ref{2.6}) in the sence of distributions. Thus,
it is sufficient to prove estimates (\ref{2.9}) and (\ref{2.9a}). Put
$$
{\cal T}_0(x,y;t,s)=\frac{|x'|^{\mu-1}}{|y'|^{\mu+1}}\Gamma^{\cal D}_{\cal K}(x,y;t,s);\quad
{\cal T}_1(x,y;t,s)=\frac{|x'|^{\mu-1}}{|y'|^\mu}D_{y}\Gamma^{\cal D}_{\cal K}(x,y;t,s);
$$
$$
{\cal T}_2(x,y;t,s)=\frac{|x'|^{\mu}}{|y'|^{\mu+1}} D_x\Gamma^{\cal D}_{\cal K}(x,y;t,s);\
{\cal T}_3(x,y;t,s)=\frac{|x'|^\mu}{|y'|^\mu} D_xD_y\Gamma^{\cal D}_{\cal K}(x,y;t,s).
$$
We choose $0<\lambda^+<\lambda_c^+$ and $0<\lambda^-<\lambda_c^-$ such that
\begin{equation}\label{mueps}
2-\frac mp-\lambda^+<\mu<m-\frac mp+\lambda^-.
\end{equation}

\smallskip

(i) By Proposition \ref{Green_D}, the kernels ${\cal T}_0$ and ${\cal T}_1$ satisfy the conditions of
Proposition \ref{L_p} with $\varepsilon_1=\varepsilon_2=0$ and

\noindent with $r=2$, $\lambda_1=\lambda^+-2$, $\lambda_2=\lambda^-$ and $\mu$ replaced by $\mu+1$ for the kernel ${\cal T}_0$;

\noindent with $r=1$, $\lambda_1=\lambda^+-1$, $\lambda_2=\lambda^--1$ for the kernel ${\cal T}_1$, respectively.\smallskip

Similarly to the first part of the proof of Theorem \ref{whole}, using Proposition \ref{L_p} and
generalized Riesz--Thorin theorem, we conclude that the operators ${\cal T}_0$ and ${\cal T}_1$ are bounded
in $\widetilde L_{p,q}(\mathbb R^n\times\mathbb R)$ for $1<p\leq q<\infty$. For $q<p$ we proceed by duality
argument and arrive at
\begin{equation}\label{2.5azz}
|\!|\!||x'|^{\mu-1}u|\!|\!|_{p,q}\leq C(|\!|\!||x'|^{\mu+1}
f_0|\!|\!|_{p,q}+|\!|\!||x'|^\mu {\bf f}|\!|\!|_{p,q}).
\end{equation}
for all $1<p,q<\infty$ and $\mu$ subject to (\ref{mu}).

To estimate the first term in the left-hand side of (\ref{2.9}) we use local estimates.
For  $\xi'' \in\mathbb{R}^{n-m}$, $\rho>0$ and $\vartheta>1$, we define
$${\Pi}_{\rho,\vartheta}(\xi'')= \Big\{x\in{\cal T}\, :\, \frac{\rho}{\vartheta}<
|x'| < \rho,\ |x''-\xi''|<\rho\Big\}.
$$
Localization of the estimate from \cite[Theorem 1 (i)]{KN2} with $\mu=0$ (by using a cut-off function,
which is equal to $1$ on ${\Pi}_{\rho,2}$ and $0$ outside ${\Pi}_{2\rho,8}$) gives
\begin{equation*}\label{2.11}
\int\limits_{{\Pi}_{\rho,2}(\xi)}\Big (\int\limits_{\mathbb R}|Du|^q\,dt\Big)^{\frac pq}dx\leq
C\int\limits_{{\Pi}_{2\rho,8}(\xi)}\Big (\int\limits_{\mathbb R}(|u|^q\rho^{-q}+
\rho^q|f_0|^q+|{\bf f}|^q\,dt\Big )^{\frac pq}dx
\end{equation*}
for any $\xi'' \in{\mathbb R}^{n-m}$ and $\rho>0$.

Using a proper partition of unity in ${\cal K}$, we arrive at
\begin{multline*}
\int\limits_{\cal K}\bigg(\int\limits_{\mathbb R} |Du|^q\, dt\bigg)^{p/q}|x'|^{\mu p}\, dx
\le C\bigg(\ \int\limits_{\cal K}\bigg(\int\limits_{\mathbb R}|u|^q\, dt\bigg)^{p/q}
|x'|^{\mu p-p}\, dx \\
+ \int\limits_{\cal K}\bigg(\int\limits_{\mathbb R}|{\bf f}|^q\, dt\bigg)^{p/q}|x'|^{\mu p}\, dx
+ \int\limits_{\cal K}\bigg(\int\limits_{\mathbb R}|f_0|^q\, dt\bigg)^{p/q}
|x'|^{\mu p+p}\, dx\bigg).
\end{multline*}
This immediately implies (\ref{2.9}) with regard of (\ref{2.5azz}).\medskip

(ii) To deal with the space $L_{p,q}$, we need the following lemma (compare with  the second part of the proof
of Theorem \ref{whole})

\begin{lem}\label{weak_2}
Let a function $h$ be supported in the layer $|s-s^0|\le\delta$ and
satisfy $\int h(y;s)\, ds\equiv 0$. Also let $p\in (1,\infty)$ and
$\mu$ be subject to (\ref{mu}).
Then the operators ${\cal T}_j$, $j=0,1,2,3$, satisfy
$$
\int\limits_{|t- s^0|>2\delta}\Vert ({\cal T}_jh)(\cdot; t)\Vert_p\, dt\le C\,\Vert h\Vert_{p,1},
$$
where $C$ does not depend on $\delta$ and $ s^0$.
\end{lem}

\begin{proof}
By $\int h(y; s)\, ds\equiv 0$, we have
\begin{equation}\label{difference1}
({\cal T}_jh)(x;t)=\int\limits_{-\infty}^{t}\int\limits_{\mathbb R^n}
\Bigl({\cal T}_j(x,y;t, s)-{\cal T}_j(x,y;t,s^0)\Bigr)\, h(y; s)\, dyds
\end{equation}
(we recall that all functions are assumed to be extended by zero outside ${\cal K}$).

We choose $0<\lambda^+<\lambda_c^+$ and $0<\lambda^-<\lambda_c^-$ such that (\ref{mueps}) holds.
Then, for $| s- s^0|<\delta$ and $t- s^0>2\delta$, estimate (\ref{May4})
implies
\begin{eqnarray*}
&&\left|{\cal T}_j(x,y;t, s)-{\cal T}_j(x,y;t, s^0)\right|
\le\int\limits_{s^0}^s|\partial_\tau {\cal T}_j(x,y;t,\tau)|\,d\tau\\
&&\le C\,\frac {{\cal R}^{\ell_1}_{x} {\cal R}^{\ell_2-2}_{y}|x'|^{\ell_3}}
{(t- s)^{\frac {n+2-r}2}|y'|^{\ell_4}r_y^{\ell_5}}\,
 \frac {\delta} {t- s}\, \exp \left(-\frac {\sigma|x-y|^2}{t- s}\right),
\end{eqnarray*}
with the following parameters (here $\varepsilon>0$ is arbitrary small):

\begin{center}
 \begin{tabular}{|c|c|c|c|c|c|c|}
 \hline \hfil Kernel\vphantom{$\left(\frac {b^b}{b^b}\right)^2$}\hfil & $r$ & $\ell_1$ & $\ell_2$ & $\ell_3$ & $\ell_4$ & $\ell_5$ \\
 \hline ${\cal T}_0$ \vphantom{$\left(\frac {b^b}{b^b}\right)^2$}
 & $2$ & $\lambda^+$ & $\lambda^-$ & $\mu-1$ & $\mu+1$ & $\varepsilon$\\
 \hline ${\cal T}_1$ \vphantom{$\left(\frac {b^b}{b^b}\right)^2$}
 & $1$ & $\lambda^+$ & $\lambda^--1$ & $\mu-1$ & $\mu$ & $1+\varepsilon$\\
 \hline ${\cal T}_2$ \vphantom{$\left(\frac {b^b}{b^b}\right)^2$}
 & $1$ & $\lambda^+-1$ & $\lambda^-$ & $\mu$ & $\mu+1$ & $\varepsilon$\\
 \hline ${\cal T}_3$ \vphantom{$\left(\frac {b^b}{b^b}\right)^2$}
 & $0$ & $\lambda^+-1$ & $\lambda^--1$ & $\mu$ & $\mu$ & $1+\varepsilon$\\
 \hline
 \end{tabular}
\end{center}
\smallskip

\noindent On the other hand, estimate (\ref{May3}) leads to
\begin{equation*}
\left|{\cal T}_j(x,y;t, s)-{\cal T}_j(x,y;t, s^0)\right| \le C\,\frac
{{\cal R}^{\ell_1}_{x}{\cal R}^{\ell_2}_{y}|x'|^{\ell_3}}{(t-s)^{\frac {n+2-r}2}|y'|^{\ell_4}}
\, \exp\left(-\frac{\sigma|x-y|^2}{t-s} \right)
\end{equation*}
with the same parameters.

Combination of these estimates gives
\begin{equation*}
\left|{\cal T}_j(x,y;t, s)-{\cal T}_j(x,y;t, s^0)\right|
\le \frac {C\delta^\varepsilon\,{\cal R}^{\ell_1}_{x} {\cal R}^{\ell_2-2\varepsilon}_{y}|x'|^{\ell_3}}
{(t- s)^{\frac {n+2-r}2+\varepsilon}|y'|^{\ell_4}r_y^{\varepsilon\ell_5}}\,\exp \left(-\frac {\sigma|x-y|^2}{t- s}\right).
\end{equation*}
Thus, the kernels in (\ref{difference1}) satisfy the assumptions of Proposition \ref{L_p_1} with $\varkappa=\varepsilon$, $\varepsilon_1=0$ and

with $\lambda_1=\lambda^+-2$, $\lambda_2=\lambda^--2\varepsilon$, $\varepsilon_2=\varepsilon^2$ and $\mu$ replaced by $\mu+1$ for kernels
${\cal T}_0$ and ${\cal T}_2$;

with $\lambda_1=\lambda^+-1$, $\lambda_2=\lambda^--1-2\varepsilon$, $\varepsilon_2=\varepsilon(1+\varepsilon)$ for kernels ${\cal T}_1$ and ${\cal T}_3$,
respectively.\smallskip

\noindent Inequality (\ref{mueps}) becomes (\ref{mu_m}), and the Lemma follows.\end{proof}

\medskip

We continue the proof of the second statement of Theorem \ref{prop11}.
Estimate (\ref{2.9}) for $q=p$ provides boundedness of the
operators ${\cal T}_j$, $j=0,1,2,3$, in $L_p(\mathbb R^n\times \mathbb R)$,
which gives the first condition in \cite[Theorem 3.8]{BIN}. Lemma \ref{weak_2} is equivalent to the
second condition in this theorem. Therefore, Theorem 3.8 \cite{BIN}
ensures that these operators are bounded in $L_{p,q}(\mathbb R^n\times \mathbb R)$ for any
 $1<q<p<\infty$. For $q>p$ this statement follows by duality arguments.
This implies estimate (\ref{2.9a}). The proof is complete.

\section{Oblique derivative problem}\label{Neu}

\subsection{The Green function}

From here on we use the notation $x'=(x_1,\hat{x})$ and assume that the cone $K$ is strictly Lipschitz, i.e. (\ref{LUD1}) and (\ref{LUD2}) are satisfied.

\begin{sats} \label{Th2}
There exists a Green function
$\Gamma_{\cal K}^{\cal N}=\Gamma_{\cal K}^{\cal N}(x,y;t,s)$ of problem
{\rm (\ref{2.6zz})}. Moreover, if
$\lambda^+<\lambda_c^+$ and $\lambda^-<\lambda_c^-$ then
the following estimates are valid 
for arbitrary $x,y\in{\cal K}$, $ t>s$:
\begin{eqnarray}\label{Ap1a}
&&|D^\alpha_xD^\beta_y\Gamma_{\cal K}^{\cal N}(x,y;t,s)|\leq
C\,\frac{{\cal R}_{x,t-s}^{(\lambda^+-|\alpha'|+1)_-}}{r_x^{(|\alpha'|-3+\varepsilon)_+}}
\,\frac{{\cal R}_{y,t-s}^{\lambda^--|\beta'|-1}}{r_y^{(|\beta'|-1+\varepsilon)_+}}
\\
&&
\times\frac{(1-{\cal R}_{x,t-s})^{-\min\{1,(|\alpha'|-2+\varepsilon)_+\}}}{(t-s)^{\frac {n+|\alpha|+|\beta|}2}}
\,\exp \left(-\sigma\,\frac{|\hat{x}-\hat{y}|^2+|x''-y''|^2}{t-s}\right),
\nonumber
\end{eqnarray}
\begin{eqnarray}\label{May34}
&&|D_{x}^{\alpha}D_{y}^{\beta}\partial_s \Gamma_{\cal K}^{\cal N}(x,y;t,s)|\le
C\,\frac{{\cal R}_{x,t-s}^{(\lambda^+-|\alpha'|+1)_-}}{r_x^{(|\alpha'|-3+\varepsilon)_+}}
\,\frac{{\cal R}_{y,t-s}^{\lambda^--|\beta'|-3}}{r_y^{|\beta'|+1+\varepsilon}}
\\
&&
\times\frac{(1-{\cal R}_{x,t-s})^{-\min\{1,(|\alpha'|-2+\varepsilon)_+\}}}{(t-s)^{\frac {n+|\alpha|+|\beta|+2}2}}
\,\exp \left(-\sigma\,\frac{|\hat{x}-\hat{y}|^2+|x''-y''|^2}{t-s}\right).
\nonumber
\end{eqnarray}
Here $\sigma$ is a positive number depending on $\nu$, $\alpha$ and $\beta$,
$\varepsilon$ is an arbitrary small positive number  and $C$ may
depend on $\nu$, $\lambda^{\pm}$, $\Lambda$, $\alpha$, $\beta$ and $\varepsilon$.
\end{sats}

\begin{proof} Let $u$ be a solution of problem {\rm (\ref{2.6zz})}.
Then the derivative $D_1u$ obviously satisfies the Dirichlet problem
(\ref{2.6}) with $f_0=0$ and ${\bf f}=(f,0,\ldots,0)$.
Therefore,
$$
D_1u=-\int\limits_{-\infty}^t\int\limits_{\cal K}
D_{y_1}\Gamma_{\cal K}^{\cal D}(x,y;t,s)f(y;s)\,dyds,
$$
and we can write solution to problem (\ref{2.6zz}) as
\begin{equation}\label{2.11ag}
u(x;t)=\int\limits_{-\infty}^t\int\limits_{\cal K}\Gamma_{\cal K}^{\cal N}(x,y;t,s)f(y;s)\,dyds,
\end{equation}
where
\begin{equation}\label{Green_N}
\Gamma_{\cal K}^{\cal N}(x,y;t,s)=\int\limits_{x_1}^\infty D_{y_1}
\Gamma_{\cal K}^{\cal D}((\zeta,\hat{x},x''),y;t,s)\,d\zeta.
\end{equation}

Since $D_{x_1}\Gamma_{\cal K}^{\cal N}(x,y;t,s)=-D_{y_1}\Gamma_{\cal K}^{\cal D}(x,y;t,s)$,
we derive from (\ref{May3}) that
\begin{eqnarray*}
&&|D_{x}^{\alpha}D_{y}^{\beta} D_{x_1}\Gamma_{\cal K}^{\cal N}(x,y;t,s)|\le
C\,\frac{{\cal R}_{x,t-s}^{\lambda^+-|\alpha'|}}{r_x^{(|\alpha'|-2+\varepsilon)_+}}
\ \frac {{\cal R}_{y,t-s}^{\lambda^--|\beta'|-1}}{r_y^{(|\beta'|-1+\varepsilon)_+}}
\nonumber\\
&&\times{(t-s)^{-\frac{n+|\alpha|+|\beta|+1}2}} \,\exp
\left(-\frac{\sigma|x-y|^2}{t-s}\right),
\end{eqnarray*}
which gives (\ref{Ap1a}) for $\alpha_1\geq 1$.

To estimate derivatives with $\alpha_1=0$ we write 
\begin{eqnarray*}
&&|D_{x}^{\alpha}D_{y}^{\beta}\Gamma_{\cal K}^{\cal N}(x,y;t,s)|\le
\int\limits_{x_1}^\infty |D_{x}^{\alpha}D_{y}^{\beta}D_{y_1}
\Gamma_{\cal K}^{\cal D}((\zeta,\hat{x},x''),y;t,s)|\,d\zeta
\nonumber\\
&&\le C\, \frac {{\cal R}_{y,t-s}^{\lambda^--|\beta'|-1}}{r_y^{(|\beta'|-1+\varepsilon)_+}}\,
{(t-s)^{-\frac{n+|\alpha|+|\beta|}2}} \,
\exp\left(-\sigma\frac{|\hat{x}-\hat{y}|^2+|x''-y''|^2}{t-s}\right)
\nonumber\\
&&\times\int\limits_{x_1}^{\infty}\Big(\frac{|\hat{x}|^2+\zeta^2}{|\hat{x}|^2+\zeta^2+(t-s)}\Big)^{\frac{\lambda^+-|\alpha'|}2}
\Big(\frac {|\hat{x}|^2+\zeta^2}{(\zeta-\phi(\hat{x}))^2}\Big)^{\frac{(|\alpha'|-2+\varepsilon)_+}2}\\
&&\times\,\exp\left(-\frac{\sigma|\zeta-y_1|^2}{t-s}\right)\,\frac {d\zeta}{\sqrt{t-s}},
\end{eqnarray*}
where we have used (\ref{Green_N}) and (\ref{May3}).
Now we apply Lemma \ref{Zhut'} (see Appendix) and estimate the last integral by
$$C\,\frac{{\cal R}_{x,t-s}^{(\lambda^+-|\alpha'|+1)_-}}{r_x^{(|\alpha'|-3+\varepsilon)_+}}
\,\Big(\frac{\sqrt{t-s}}{|x'|+\sqrt{t-s}}\Big)^{-\min\{1,(|\alpha'|-2+\varepsilon)_+\}},
$$
which gives (\ref{Ap1a}) for $\alpha_1=0$.

Since $\partial_s\Gamma_{\cal K}^{\cal N}(x,y;t,s)$ can be expressed via $D^2_y\Gamma_{\cal K}^{\cal N}(x,y;t,s)$,
the estimate (\ref{May34}) follows from (\ref{Ap1a}).
\end{proof}

\subsection{Estimates for the difference $\Gamma_{\cal K}^{\cal N}-\Gamma$}

The next representation of the Green function $\Gamma_{\cal K}^{\cal N}$ follows from (\ref{Green_N}):
\begin{equation*}
\Gamma_{\cal K}^{\cal N}(x,y;t,s)=\Gamma(x,y;t,s)+\int\limits_{x_1}^\infty {\cal F}((\zeta,\hat{x},x''),y;t,s)\,d\zeta,
\end{equation*}
where $\Gamma$ is the Green function of the operator ${\cal L}$ in the whole space and
$$
{\cal F}(x,y;t,s)=D_{x_1}\Gamma(x,y;t,s)+D_{y_1}\Gamma^{\cal D}_{\cal K}(x,y;t,s).
$$
The function ${\cal F}$ verifies the Dirichlet problem
\begin{equation*}
\aligned
{\cal L}{\cal F}(x,y;t,s) & = 0 &&\mbox{in} \qquad {\cal K}\times\mathbb R;\\
{\cal F}(x,y;t,s) & = D_{x_1}\Gamma(x,y;t,s) && \mbox{for} \quad
x\in\partial {\cal K}.
\endaligned
\end{equation*}
Solving this problem, we find
$$
{\cal F}(x,y;t,s)=-\int\limits_s^t\int\limits_{\partial\cal K}
{\cal P}(x,z;t,\tau)D_{z_1}\Gamma(z,y;\tau,s)\,dS_z d\tau,
$$
where
$$
{\cal P}(x,z;t,\tau)=n_i(z)a^{ij}(\tau)D_{z_j}\Gamma_{\cal K}^{\cal D}(x,z;t,\tau).
$$
Thus we arrived at the following representation for the difference
$\Gamma_{\cal K}^{\cal N}-\Gamma$:
\begin{eqnarray}\label{difference}
&&\Gamma_{\cal K}^{\cal N}(x,y;t,s)-\Gamma(x,y;t,s)
\\
&&=-\int\limits_{x_1}^\infty
\int\limits_s^t\int\limits_{\partial\cal K} {\cal
P}((\zeta,\hat{x},x''),z;t,\tau)D_{z_1}\Gamma(z,y;\tau,s)\,dS_z
d\tau d\zeta.
\nonumber
\end{eqnarray}

Now we are in the position to prove the main estimate of this Section.

\begin{sats} \label{Th3}
Let $\lambda^+<\lambda_c^+$ and $\lambda^-<\lambda_c^-$. Then
the following estimates are valid for $|\alpha|=2$ and for arbitrary $x,y\in{\cal K}$, $ t>s$:
\begin{eqnarray}\label{Feb23}
&&|D_{x}^{\alpha}D_{y}^{\beta}\big(\Gamma_{\cal K}^{\cal N}(x,y;t,s)-\Gamma(x,y;t,s)\big)|
\\
&&\leq\frac{C\,{\cal R}_{x,t-s}^{(\lambda^+-2)_-}}
{(t-s)^{\frac{n+2+|\beta|}2}}
\exp\Big(-\frac{\widetilde\sigma |x-y|^2}{t-s}\Big)
\nonumber\\
&&\times\bigg[\frac{(t-s)^{\frac 12+\varepsilon}}
{|x'|^{1+\varepsilon}\,r_x^{1+3\varepsilon}\,|y'|^{\varepsilon}\,r_y^{\varepsilon}}
+\frac{{\cal R}_{y,t-s}^{(\lambda^--1)_--1}(t-s)^{\frac{|\beta|}2+1}}
{|x'|\,r_x^{1+2\varepsilon}\,|y'|^{|\beta|+1}\,r_y^{|\beta|+2}}\bigg],
\nonumber
\end{eqnarray}
\begin{eqnarray}\label{Feb23a}
&&|D_{x}^{\alpha}D_{y}^{\beta}\partial_s\big(\Gamma_{\cal K}^{\cal N}(x,y;t,s)-\Gamma(x,y;t,s)\big)|
\\
&&\leq\frac{C\,{\cal R}_{x,t-s}^{(\lambda^+-2)_-}}
{(t-s)^{\frac{n+4+|\beta|}2}}
\exp\Big(-\frac{\widetilde\sigma |x-y|^2}{t-s}\Big)
\nonumber\\
&&\times\bigg[\frac{(t-s)^{\frac 12+\varepsilon}}
{|x'|^{1+\varepsilon}\,r_x^{1+3\varepsilon}\,|y'|^{\varepsilon}\,r_y^{\varepsilon}}
+\frac{{\cal R}_{y,t-s}^{(\lambda^--1)_--1}(t-s)^{\frac{|\beta|}2+2}}
{|x'|\,r_x^{1+2\varepsilon}\,|y'|^{|\beta|+3}\,r_y^{|\beta|+4}}\bigg].
\nonumber
\end{eqnarray}
Here $\widetilde\sigma$ is a positive number depending on $\nu$, $\Lambda$ and $\beta$,
$\varepsilon$ is an arbitrary small positive number and $C$ may
depend on $\nu$, $\lambda^{\pm}$, $\Lambda$, $\beta$ and $\varepsilon$.
\end{sats}

\begin{proof}
We begin with estimates for $D_{x}^{\alpha}D_{y}^{\beta}{\cal F}$ with arbitrary $\alpha$. Using (\ref{May3}) and (\ref{May0}),
we obtain
\begin{eqnarray*}
&&|D_{x}^{\alpha}D_{y}^{\beta}{\cal F}(x,y;t,s)|\leq C
\int\limits_s^t
\int\limits_{\partial\cal K}\frac{{\cal R}_{x,t-\tau}^{\lambda^+-|\alpha'|}\,r_x^{-(|\alpha'|-2+\varepsilon)_+}\,
{\cal R}_{z,t-\tau}^{\lambda^--1}}
{(t-\tau)^{\frac{n+|\alpha|+1}2}(\tau-s)^{\frac{n+|\beta|+1}2}}
\nonumber\\
&&\times\exp\Big(-\frac{\sigma|x-z|^2}{t-\tau}\Big)
\exp\Big(-\frac{\sigma|z-y|^2}{\tau-s}\Big)\,dS_zd\tau,
\end{eqnarray*}

Taking into account that $z=(\phi(\hat{z}),\hat{z},z'')$ and integrating with respect to $z''$, we arrive at
\begin{eqnarray}\label{23okt}
&&|D_{x}^{\alpha}D_{y}^{\beta}{\cal F}(x,y;t,s)|\leq \frac {C\,r_x^{-(|\alpha'|-2+\varepsilon)_+}}
{(t-s)^{\frac{n-m}2}}
\nonumber\\
&&\times\int\limits_s^t\frac{{\cal R}_{x,t-\tau}^{\lambda^+-|\alpha'|}}
{(t-\tau)^{\frac{m+|\alpha|+1}2}}\,
\int\limits_{\mathbb R^{m-1}}\exp\Big(-\sigma\,\frac{|\hat{x}-\hat{z}|^2+|x_1-\phi(\hat{z})|^2}{t-\tau}\Big)
\nonumber\\
&&\times\frac{{\cal R}_{z,t-\tau}^{\lambda^--1}}{(\tau-s)^{\frac{m+|\beta|+1}2}}\,
\exp\Big(-\frac{\sigma|z'-y'|^2}{\tau-s}\Big)\,d\hat{z}d\tau.
\end{eqnarray}

Let us estimate the expression in the last line of (\ref{23okt}). If $|z'-y'|\le |y'|$
then $|z'|\le 2|y'|$ and therefore
${\cal R}_{z,t-\tau}^{\lambda^--1}\le{\cal R}_{y,t-\tau}^{(\lambda^--1)_-}$. In the opposite case
we have $|z'-y'|^2\ge\frac 12 (|\hat{z}-\hat{y}|^2+|y'|^2)$. In any case we obtain
\begin{eqnarray}\label{23oktc}
{\cal R}_{z,t-\tau}^{\lambda^--1}\exp\Big(-\frac{\sigma|z'-y'|^2}{\tau-s}\Big)
&\le & {\cal R}_{y,t-\tau}^{(\lambda^--1)_-}
\exp\Big(-\sigma\,\frac{|\hat{z}-\hat{y}|^2+|y_1-\phi(\hat{z})|^2}{\tau-s}\Big)
\nonumber\\
&+& {\cal R}_{z,t-\tau}^{\lambda^--1}
\exp\Big(-\sigma\,\frac{|\hat{z}-\hat{y}|^2+|y'|^2}{\tau-s}\Big).
\end{eqnarray}

Next, we claim that the following inequality
\begin{equation}\label{K1}
|\hat{x}-\hat{z}|^2+c|x_1-\phi(\hat{z})|^2\geq
(1-c\Lambda^2)|\hat{x}-\hat{z}|^2+\frac{c}{2}|x_1-\phi(\hat{x})|^2,
\end{equation}
holds for positive $c$. Indeed, we note that
$$
|x_1-\phi(\hat{x})|^2\leq
2|x_1-\phi(\hat{z})|^2+2|\phi(\hat{x})-\phi(\hat{z})|^2\leq
2|x_1-\phi(\hat{z})|^2+2\Lambda^2|\hat{x}-\hat{z}|^2,
$$
which implies
\begin{equation*}
|x_1-\phi(\hat{z})|^2\geq
\frac{1}{2}|x_1-\phi(\hat{x})|^2-\Lambda^2|\hat{x}-\hat{z}|^2
\end{equation*}
and leads to (\ref{K1}).

Using (\ref{23oktc}) and (\ref{K1}), we estimate the integral in (\ref{23okt}) by
$$\int\limits_s^t\frac{{\cal R}_{x,t-\tau}^{\lambda^+-|\alpha'|}}
{(t-\tau)^{\frac{m+|\alpha|+1}2}(\tau-s)^{\frac{m+|\beta|+1}2}}\,
\exp\Big(-\frac{\sigma_1|x_1-\phi(\hat{x})|^2}{t-\tau}\Big)
\cdot\Big({\cal I}_1+{\cal I}_2\Big)\,d\tau,
$$
where
\begin{eqnarray}\label{23okta}
&&
{\cal I}_1={\cal R}_{y,t-\tau}^{(\lambda^--1)_-}
\exp\Big(-\frac{\sigma_1|y_1-\phi(\hat{y})|^2}{\tau-s}\Big)\\
&&\times\int\limits_{\mathbb R^{m-1}}
\exp\Big(-\frac{\sigma_1|\hat{x}-\hat{z}|^2}{t-\tau}\Big)
\exp\Big(-\frac{\sigma_1|\hat{z}-\hat{y}|^2}{\tau-s}\Big)\,d\hat{z}.
\nonumber
\end{eqnarray}
\begin{eqnarray}\label{23oktd}
&&
{\cal I}_2=\exp\Big(-\frac{\sigma_1|y'|^2}{\tau-s}\Big)\\
&&\times\int\limits_{\mathbb R^{m-1}}
\Big(\frac{|\hat{z}|}{|\hat{z}|+\sqrt{t-\tau}}\Big)^{\lambda^--1}
\exp\Big(-\frac{\sigma_1|\hat{x}-\hat{z}|^2}{t-\tau}\Big)
\exp\Big(-\frac{\sigma_1|\hat{z}-\hat{y}|^2}{\tau-s}\Big)\,d\hat{z}.
\nonumber
\end{eqnarray}

Applying Lemma \ref{Lozenka1} (see Appendix) for estimating
the right-hand side of (\ref{23okta}) and (\ref{23oktd}), we get
\begin{eqnarray}\label{23oktb}
&&|D_{x}^{\alpha}D_{y}^{\beta}{\cal F}(x,y;t,s)|\leq
\frac{C\,r_x^{-(|\alpha'|-2+\varepsilon)_+}}{(t-s)^{\frac{n-1}2}}
\,\exp\Big(-\sigma_1\,\frac{|\hat{x}-\hat{y}|^2+|x''-y''|^2}{t-s}\Big)
\nonumber\\
&&\times\int\limits_s^t\frac{{\cal R}_{x,t-\tau}^{\lambda^+-|\alpha'|}}
{(t-\tau)^{1+\frac{|\alpha|}2}(\tau-s)^{1+\frac{|\beta|}2}}
\,\exp\Big(-\frac{\sigma_1|x_1-\phi(\hat{x})|^2}{t-\tau}\Big)
\\
&&\times\bigg[\Big(\frac {\tau-s}{t-s}\Big)^{\frac{\varrho}2}
\exp\Big(-\frac{\sigma_1|y'|^2}{\tau-s}\Big)
+{\cal R}_{y,t-\tau}^{\varrho}
\exp\Big(-\frac{\sigma_1|y_1-\phi(\hat{y})|^2}{\tau-s}\Big)\bigg]\,d\tau,
\nonumber
\end{eqnarray}
where $\varrho=(\lambda^--1)_-$.

Now we note that ${\cal R}_{x,t-\tau}^{\lambda^+-|\alpha'|}\le {\cal R}_{x,t-s}^{(\lambda^+-|\alpha'|)_-}$,
${\cal R}_{y,t-\tau}^{\varrho}\le {\cal R}_{y,t-s}^{\varrho}$,
and make the following change of variables in the right-hand side of (\ref{23oktb}):
\begin{equation*}
\theta=\frac{t-\tau}{\tau-s}, \qquad \frac {d\tau}{d\theta}=-\frac{t-s}{(\theta+1)^2}.
\end{equation*}
This gives
\begin{eqnarray}\label{14feb3}
&&|D_{x}^{\alpha}D_{y}^{\beta}{\cal F}(x,y;t,s)|\leq
\frac{C\,{\cal R}_{x,t-s}^{(\lambda^+-|\alpha'|)_-}\,r_x^{-(|\alpha'|-2+\varepsilon)_+}}
{(t-s)^{\frac{n+|\alpha|+|\beta|+1}2}}
\\
&&\times\exp\Big(-\sigma_2\,
\frac{|\hat{x}-\hat{y}|^2+|x''-y''|^2+|x_1-\phi(\hat{x})|^2+|y_1-\phi(\hat{y})|^2}{t-s}\Big)
\nonumber\\
&&\times\int\limits_0^{\infty}d\theta\,\exp\Big(-\frac{\sigma_1|x_1-\phi(\hat{x})|^2}
{(t-s)\theta}\Big)
\frac{(\theta+1)^{\frac{|\alpha|+|\beta|}2}}{\theta^{1+\frac{|\alpha|}2}}
\nonumber\\
&&\times\bigg[(\theta+1)^{-\frac{\varrho}2}\!\exp\Big(-\frac{\sigma_1|y'|^2\theta}{t-s}\Big)+
{\cal R}_{y,t-s}^{\varrho} \exp\Big(-\frac{\sigma_1|y_1-\phi(\hat{y})|^2\theta}{t-s}\Big)\bigg]
\nonumber
\end{eqnarray}
(we take into account that $|y_1-\phi(\hat{y})|\le C |y'|$).

Using formula
$$\int\limits_0^{\infty}\theta^{\kappa-1}\exp(-A\theta-B\theta^{-1})\,d\theta=
2(BA^{-1})^{\frac{\kappa}2}\cdot {\bf K}_{\kappa}(2\sqrt{AB})
$$
(see, e.g., \cite[3.471.9]{GR}) and estimates of the Macdonald function (\cite[8.407 and 8.451.6]{GR})
$${\bf K}_{\kappa}(\zeta)\le C(\kappa,\gamma)\, \zeta^{-(|\kappa|+\gamma)}, \qquad\qquad \forall\zeta,\gamma>0,
$$
we estimate the integral in (\ref{14feb3}) by
\begin{eqnarray}\label{23feb}
&&C\,\bigg[\frac{|x_1-\phi(\hat{x})|^{-\varepsilon}|y_1-\phi(\hat{y})|^{-\varepsilon}}{(t-s)^{-\varepsilon}}
\,\Big(\frac{|x_1-\phi(\hat{x})|^2}{t-s}\Big)^{-\frac{|\alpha|}2}
\\
&&+\frac{|x_1-\phi(\hat{x})|^{-\gamma}|y_1-\phi(\hat{y})|^{-\gamma}}{(t-s)^{-\gamma}}
\,\bigg[\Big(\frac{|y'|^2}{t-s}\Big)^{\frac{\varrho-|\beta|}2}
+{\cal R}_{y,t-s}^{\varrho}\Big(\frac{|y_1-\phi(\hat{y})|^2}{t-s}\Big)^{-\frac{|\beta|}2}\bigg]\bigg]
\nonumber\\
&&\le C\,\bigg[\frac{(t-s)^{\frac{|\alpha|}2+\varepsilon}}
{|x'|^{|\alpha|+\varepsilon}\,r_x^{|\alpha|+\varepsilon}\,|y'|^{\varepsilon}\,r_y^{\varepsilon}}+
\frac{{\cal R}_{y,t-s}^{\varrho}(t-s)^{\frac{|\beta|}2+\gamma}}
{|x'|^{\gamma}\,r_x^{\gamma}\,|y'|^{|\beta|+\gamma}\,r_y^{|\beta|+\gamma}}\bigg].
\nonumber
\end{eqnarray}

Next, we note that
$$
\aligned
|x_1-y_1|^2
&\leq 3|x_1-\phi(\hat{x})|^2+3|\phi(\hat{x})-\phi(\hat{y})|^2+3|\phi(\hat{y})-y_1|^2\\
&\leq 3|x_1-\phi(\hat{x})|^2+3|y_1-\phi(\hat{y})|^2+3\Lambda^2|\hat{x}-\hat{y}|^2,
\endaligned
$$
which implies
\begin{equation}\label{23feb1}
|\hat{x}-\hat{y}|^2+|x_1-\phi(\hat{x})|^2+|y_1-\phi(\hat{y})|^2\ge
\frac{1}{3(1+\Lambda^2)}\,|x'-y'|^2.
\end{equation}

For $|\alpha|=1$ we derive from (\ref{14feb3}), (\ref{23feb}) and (\ref{23feb1}) that
\begin{eqnarray*}
&&|D_{x}^{\alpha}D_{y}^{\beta}D_{x_1}\big(\Gamma_{\cal K}^{\cal N}(x,y;t,s)-\Gamma(x,y;t,s)\big)|
\\
&&\leq\frac{C\,{\cal R}_{x,t-s}^{(\lambda^+-1)_-}
}
{(t-s)^{\frac{n+2+|\beta|}2}}
\exp\Big(-\frac{\widetilde\sigma |x-y|^2}{t-s}\Big)
\\
&&\times\bigg[\frac{(t-s)^{\frac 12+\varepsilon}}
{|x'|^{1+\varepsilon}\,r_x^{1+\varepsilon}\,|y'|^{\varepsilon}\,r_y^{\varepsilon}}+
\frac{{\cal R}_{y,t-s}^{(\lambda^--1)_-}(t-s)^{\frac{|\beta|}2+\gamma}}
{|x'|^{\gamma}\,r_x^{\gamma}\,|y'|^{|\beta|+\gamma}\,r_y^{|\beta|+\gamma}}\bigg].
\end{eqnarray*}
Setting $\gamma=1$, we obtain (\ref{Feb23}) for $\alpha_1\ge1$.

To estimate derivatives for $|\alpha|=2$, $\alpha_1=0$, we use  (\ref{difference}), (\ref{14feb3})
and (\ref{23feb}) and write 
\begin{eqnarray*}
&&|D_{x}^{\alpha}D_{y}^{\beta}\big(\Gamma_{\cal K}^{\cal N}(x,y;t,s)-\Gamma(x,y;t,s)\big)|
\\
&&\le\int\limits_{x_1}^\infty |D_{x}^{\alpha}D_{y}^{\beta}
{\cal F}((\zeta,\hat{x},x''),y;t,s)|\,d\zeta
\nonumber\\
&&\le\frac{C}{(t-s)^{\frac{n+2+|\beta|}2}}\,\cdot\Big({\cal J}_1+{\cal J}_2\Big)
\\
&&\times\exp\Big(-\sigma_2\,
\frac{|\hat{x}-\hat{y}|^2+|x''-y''|^2+|x_1-\phi(\hat{x})|^2+|y_1-\phi(\hat{y})|^2}{t-s}\Big),
\nonumber\\
&&\le\frac{C}{(t-s)^{\frac{n+2+|\beta|}2}}\,\cdot\Big({\cal J}_1+{\cal J}_2\Big)
\exp\Big(-\frac{\widetilde\sigma |x-y|^2}{t-s}\Big)
\end{eqnarray*}
(the last inequality follows from (\ref{23feb1})), where
\begin{eqnarray*}
&{\cal J}_1&=\frac{(t-s)^{\frac{\varepsilon}2}}{|y'|^{\varepsilon}\,r_y^{\varepsilon}}
\int\limits_{x_1}^{\infty}
\exp\left(-\frac{\sigma|\zeta-x_1|^2}{t-s}\right)\,
\Big(\frac{|\hat{x}|^2+\zeta^2}{|\hat{x}|^2+\zeta^2+(t-s)}\Big)^{\frac{(\lambda^+-2)_-}2}
\\
&&\times\,
\Big(\frac {|\hat{x}|^2+\zeta^2}{(\zeta-\phi(\hat{x}))^2}\Big)^{\frac{\varepsilon}2}
\Big(\frac {t-s}{(\zeta-\phi(\hat{x}))^2}\Big)^{\frac{1+\varepsilon}2}
\frac {d\zeta}{\sqrt{t-s}},
\end{eqnarray*}
\begin{eqnarray*}
&{\cal J}_2&=\frac{{\cal R}_{y,t-s}^{(\lambda^--1)_-}(t-s)^{\frac{|\beta|+\gamma}2}}
{|y'|^{|\beta|+\gamma}\,r_y^{|\beta|+\gamma}}
\int\limits_{x_1}^{\infty}
\Big(\frac{|\hat{x}|^2+\zeta^2}{|\hat{x}|^2+\zeta^2+(t-s)}\Big)^{\frac{(\lambda^+-2)_-}2}\\
&&\times
\Big(\frac {|\hat{x}|^2+\zeta^2}{(\zeta-\phi(\hat{x}))^2}\Big)^{\frac{\varepsilon}2}
\Big(\frac {t-s}{(\zeta-\phi(\hat{x}))^2}\Big)^{\frac{\gamma}2}
\exp\left(-\frac{\sigma|\zeta-x_1|^2}{t-s}\right)\,\frac {d\zeta}{\sqrt{t-s}}.
\end{eqnarray*}
We set $\gamma=2$ and estimate these integrals by using Lemma \ref{Zhut'}. This again gives (\ref{Feb23}).

Since $\partial_s(\Gamma_{\cal K}^{\cal N}-\Gamma)$ can be expressed via $D^2_y(\Gamma_{\cal K}^{\cal N}-\Gamma)$,
 estimate (\ref{Feb23a}) follows from (\ref{Feb23}).
\end{proof}

\subsection{Coercive estimates in $\widetilde{L}_{p,q}$ and  $L_{p,q}$}\label{Sect3.3}

\begin{lem}\label{prop12} Let $\lambda_c^{\pm}$ be the critical exponents and $1<p,q<\infty$. Let also $\mu$ be subject to {\rm (\ref{mu1})}.
Then the integral operators with kernels
$$
{\cal T}_4(x,y;t,s)=\frac{|x'|^\mu}{|y'|^\mu}D_xD_x\big({\Gamma}_{\cal K}^{\cal N}(x,y;t,s)-{\Gamma}(x,y;t,s)\big)
$$
are bounded in $\widetilde{L}_{p,q}(\mathbb R^n\times\mathbb R)$.
\end{lem}

\begin{proof}
First, we choose $0<\lambda^+<\lambda_c^+$, $0<\lambda^-<\lambda_c^-$ and $0<\varepsilon<\frac 14\min\{\frac 1p,1-\frac 1p\}$ such that
\begin{equation}\label{mueps1}
-\frac mp+(1-\lambda^+)_++3\varepsilon<\mu<m-\frac mp-(1-\lambda^-)_+-2\varepsilon.
\end{equation}

Using (\ref{May0}) and (\ref{Ap1a}) with $|\alpha|=2$, $|\beta|=0$, we obtain
\begin{eqnarray}\label{Feb4}
&&|{\cal T}_4(x,y;t,s)|\leq
C\,\frac{{\cal R}_{x,t-s}^{(\lambda^+-1)_--\varepsilon}\,{\cal R}_{y,t-s}^{(\lambda^--1)_-}\,|x'|^{\mu+\varepsilon}}
{(t-s)^{\frac {n+2+\varepsilon}2}\,|y'|^{\mu}}
\\
&&\times\exp \left(-\sigma\,\frac{|\hat{x}-\hat{y}|^2+|x''-y''|^2}{t-s}\right),
\nonumber
\end{eqnarray}
while  estimate (\ref{Feb23}) with $|\beta|=0$ gives
\begin{eqnarray*}
&&|{\cal T}_4(x,y;t,s)|\leq\frac{C\,{\cal R}_{x,t-s}^{(\lambda^+-2)_-}\,|x'|^{\mu}}
{(t-s)^{\frac{n+2}2}\,|y'|^{\mu}}
\exp\Big(-\frac{\widetilde\sigma |x-y|^2}{t-s}\Big)
\\
&&\times\bigg[\frac{(t-s)^{\frac 12+\varepsilon}}
{|x'|^{1+\varepsilon}\,r_x^{1+3\varepsilon}\,|y'|^{\varepsilon}\,r_y^{\varepsilon}}
+\frac{{\cal R}_{y,t-s}^{(\lambda^--1)_--1}(t-s)}
{|x'|\,r_x^{1+2\varepsilon}\,|y'|\,r_y^2}\bigg].
\end{eqnarray*}
Combination of these estimates gives
\begin{eqnarray}\label{sum}
&&|{\cal T}_4(x,y;t,s)|\leq C\,\frac{{\cal R}_{x,t-s}^{(\lambda^+-1)_--2\varepsilon}\,{\cal R}_{y,t-s}^{(\lambda^--1)_-}
\,|x'|^{\mu-\varepsilon^2}}
{(t-s)^{\frac{n+2-2\varepsilon^2}2}\,|y'|^{\mu+\varepsilon^2}\,r_x^{\varepsilon+3\varepsilon^2}\,r_y^{\varepsilon^2}}
\exp\Big(-\frac{\varepsilon\widetilde\sigma |x-y|^2}{t-s}\Big)
\nonumber\\
&&+C\,\frac{{\cal R}_{x,t-s}^{(\lambda^+-1)_--2\varepsilon}\,{\cal R}_{y,t-s}^{(\lambda^--1)_--\varepsilon}
\,|x'|^{\mu}}
{(t-s)^{\frac{n+2-\varepsilon}2}\,|y'|^{\mu+\varepsilon}\,r_x^{\varepsilon+2\varepsilon^2}\,r_y^{2\varepsilon}}
\exp\Big(-\frac{\varepsilon\widetilde\sigma |x-y|^2}{t-s}\Big).
\end{eqnarray}

The summands in (\ref{sum}) satisfy the conditions of Proposition \ref{L_p}

\noindent with $r=2\varepsilon^2$, $\lambda_1=-(1-\lambda^+)_+-2\varepsilon(1+\varepsilon)$, $\lambda_2=-(1-\lambda^-)_+$,
$\varepsilon_1=\varepsilon(1+3\varepsilon)$, $\varepsilon_2=\varepsilon^2$ and $\mu$ replaced by $\mu+\varepsilon^2$ for the
first summand;

\noindent with $r=\varepsilon$, $\lambda_1=-(1-\lambda^+)_+-3\varepsilon$, $\lambda_2=-(1-\lambda^-)_+-\varepsilon$,
$\varepsilon_1=\varepsilon(1+2\varepsilon)$, $\varepsilon_2=2\varepsilon$ and $\mu$ replaced by $\mu+\varepsilon$ for the
second summand, respectively.\smallskip

\noindent Note that the relation $\lambda_1+\lambda_2>-m$ holds in both cases due to (\ref{mueps1}).

Similarly to the first part of the proof of Theorem \ref{whole}, using Proposition \ref{L_p} and
generalized Riesz--Thorin theorem, we conclude that the operators ${\cal T}_4$ are bounded in
$\widetilde L_{p,q}(\mathbb R^n\times\mathbb R)$ for $1<p\leq q<\infty$. For $q<p$ we proceed by duality
argument.
\end{proof}

\begin{lem}\label{prop13} Under assumptions of {\rm Lemma \ref{prop12}} the operators ${\cal T}_4$ are bounded in
$L_{p,q}(\mathbb R^n\times\mathbb R)$.
\end{lem}

\begin{proof}
We proceed as in the second part of the proof of Theorem \ref{whole}. Let a function $h$ be supported in
the layer $|s-s^0|\le\delta$ and satisfy $\int h(y;s)\, ds\equiv 0$. Then
\begin{equation*}
({\cal T}_4h)(x;t)=\int\limits_{-\infty}^{t}\int\limits_{\mathbb R^n}
\Bigl({\cal T}_4(x,y;t, s)-{\cal T}_4(x,y;t,s^0)\Bigr)\, h(y; s)\, dyds.
\end{equation*}

We choose $0<\lambda^+<\lambda_c^+$, $0<\lambda^-<\lambda_c^-$ and $0<\varepsilon<\frac 14\min\{\frac 1p,1-\frac 1p\}$ such that
\begin{equation*}\label{mueps2}
-\frac mp+(1-\lambda^+)_++3\varepsilon<\mu<m-\frac mp-(1-\lambda^-)_+-4\varepsilon.
\end{equation*}
Then, for $| s- s^0|<\delta$ and $t- s^0>2\delta$, estimate (\ref{Feb23a}) with $|\beta|=0$ implies
\begin{eqnarray*}
&&\left|{\cal T}_4(x,y;t, s)-{\cal T}_4(x,y;t, s^0)\right|
\le\int\limits_{s^0}^s|\partial_\tau {\cal T}_4(x,y;t,\tau)|\,d\tau
\\
&&\le \frac{C\,{\cal R}_{x,t-s}^{(\lambda^+-2)_-}\,|x'|^{\mu}}
{(t-s)^{\frac{n+2}2}\,|y'|^{\mu}}\,\frac {\delta}{t-s}
\exp\Big(-\frac{\widetilde\sigma |x-y|^2}{t-s}\Big)
\\
&&\times\bigg[\frac{(t-s)^{\frac 12+\varepsilon}}
{|x'|^{1+\varepsilon}\,r_x^{1+3\varepsilon}\,|y'|^{\varepsilon}\,r_y^{\varepsilon}}
+\frac{{\cal R}_{y,t-s}^{(\lambda^--1)_--1}(t-s)^2}
{|x'|\,r_x^{1+2\varepsilon}\,|y'|^3\,r_y^4}\bigg].
\end{eqnarray*}
Combining this estimate with (\ref{Feb4}), we get
\begin{eqnarray}\label{sum1}
&&\left|{\cal T}_4(x,y;t, s)-{\cal T}_4(x,y;t, s^0)\right|
\\
&&\leq C\,\frac{\delta^\varepsilon\,{\cal R}_{x,t-s}^{(\lambda^+-1)_--2\varepsilon}\,{\cal R}_{y,t-s}^{(\lambda^--1)_-}
\,|x'|^{\mu-\varepsilon^2}}
{(t-s)^{\frac{n+2-2\varepsilon^2}2+\varepsilon}\,|y'|^{\mu+\varepsilon^2}\,r_x^{\varepsilon+3\varepsilon^2}\,r_y^{\varepsilon^2}}
\exp\Big(-\frac{\varepsilon\widetilde\sigma |x-y|^2}{t-s}\Big)
\nonumber\\
&&+C\,\frac{\delta^\varepsilon\,{\cal R}_{x,t-s}^{(\lambda^+-1)_--2\varepsilon}
\,{\cal R}_{y,t-s}^{(\lambda^--1)_--\varepsilon}\,|x'|^{\mu}}
{(t-s)^{\frac{n+2-3\varepsilon}2+\varepsilon}\,|y'|^{\mu+3\varepsilon}\,r_x^{\varepsilon+2\varepsilon^2}\,r_y^{4\varepsilon}}
\exp\Big(-\frac{\varepsilon\widetilde\sigma |x-y|^2}{t-s}\Big).
\nonumber
\end{eqnarray}
The summands in the right-hand side of (\ref{sum1}) satisfy the conditions of Proposition \ref{L_p_1} with $\varkappa=\varepsilon$ and

\noindent with $r=2\varepsilon^2$, $\lambda_1=-(1-\lambda^+)_+-2\varepsilon(1+\varepsilon)$, $\lambda_2=-(1-\lambda^-)_+$,
$\varepsilon_1=\varepsilon(1+3\varepsilon)$, $\varepsilon_2=\varepsilon^2$ and $\mu$ replaced by $\mu+\varepsilon^2$ for the
first summand;

\noindent with $r=3\varepsilon$, $\lambda_1=-(1-\lambda^+)_+-5\varepsilon$, $\lambda_2=-(1-\lambda^-)_+-\varepsilon$,
$\varepsilon_1=\varepsilon(1+2\varepsilon)$, $\varepsilon_2=4\varepsilon$ and $\mu$ replaced by $\mu+3\varepsilon$ for the
second summand, respectively.\smallskip

Applying Proposition \ref{L_p_1}, we obtain the second condition in \cite[Theorem 3.8]{BIN}, while Lemma \ref{prop12}
for $q=p$ gives the first condition in this theorem. Therefore, Theorem 3.8 \cite{BIN}
ensures that the operators ${\cal T}_4$ are bounded in $L_{p,q}(\mathbb R^n\times \mathbb R)$ for any
$1<q<p<\infty$. For $q>p$ the statement follows by duality arguments.\end{proof}

Now we are in position to prove the main result of our paper.\medskip

{\bf Proof of Theorem \ref{Th1s}.}
The estimate of the last terms in the left-hand sides of
(\ref{TTn1a}) and (\ref{TTn1b}) is equivalent to the boundedness of integral
operators with kernels
$$\frac {|x'|^\mu} {|y'|^\mu}D_xD_x\Gamma^{\cal N}_{\cal K}(x,y;t,s)$$
in $\widetilde L_{p,q}(\mathbb R^n\times \mathbb R)$ and
$L_{p,q}(\mathbb R^n\times \mathbb R)$, respectively. The first statement follows from
Lemma \ref{prop12} and Theorem \ref{whole}, the second one -- form Lemma \ref{prop13} and Theorem
\ref{whole}. The first terms in (\ref{TTn1a}) and in (\ref{TTn1b}) are estimated by using equation (\ref{Jan1}),
and Theorem follows.

\section{Appendix}

\begin{lem}\label{Zhut'} Let $\phi$ be a function on $\mathbb{R}_+$ such that
$$
|\phi(\zeta)|\leq \Lambda |\zeta|\qquad\mbox{for all}
\quad\zeta \in\mathbb R_+.
$$
Let $\varrho_1,\,\varrho_2>0$, $x_1>\phi(\varrho_1)$, $y_1\in\mathbb R$, $a,\,b,\,c\ge0$. Then
\begin{eqnarray}\label{Ap1b}
&&\int\limits_{x_1}^{\infty}
\Big(\frac{\varrho_1+|\zeta|}{\varrho_1+\varrho_2+|\zeta|}\Big)^{-a}
\Big(\frac {\varrho_1+|\zeta|}{\zeta-\phi(\varrho_1)}\Big)^{b}
\Big(\frac {\varrho_2}{\zeta-\phi(\varrho_1)}\Big)^{c}
\\
&&\times\exp\left(-\frac{|\zeta-y_1|^2}{\varrho_2^2}\right)\,\frac {d\zeta}{\varrho_2}
\le C\,\Big(\frac {x_1-\phi(\varrho_1)}{\varrho_1+|x_1|}\Big)^{(1-b-c-\varepsilon)_-}
\nonumber\\
&&\times\begin{cases}
\displaystyle\vphantom{\left(\frac{|x_1|}{|x_1|}\right)^{-1}}
\Big(\frac{\varrho_1+|x_1|}{\varrho_1+\varrho_2+|x_1|}\Big)^{-a}
\Big(\frac {\varrho_1+|x_1|}{\varrho_2}\Big)^{1-c}, & {\rm if}\quad c>1;
\\
\displaystyle\vphantom{\left(\frac{|x_1|}{|x_1|}\right)^{-1}}
\Big(\frac{\varrho_1+|x_1|}{\varrho_1+\varrho_2+|x_1|}\Big)^{-(a+c+\varepsilon-1)_+}
\Big(\frac {\varrho_1+\varrho_2+|x_1|}{\varrho_2}\Big)^{\min\{b,1-c\}}, &  {\rm if}\quad c\le1.
\end{cases}
\nonumber
\end{eqnarray}
Here $\varepsilon$ is arbitrary small positive number while $C$ may depend on $a$, $b$, $c$,
$\Lambda$ and $\varepsilon$.
\end{lem}

\begin{proof}
{\it Case 1}: $x_1\ge 2\Lambda\varrho_1$. Then $\varrho_1+|\zeta|\asymp\zeta-\phi(\varrho_1)$
for $\zeta\ge x_1$, and integral is estimated by
\begin{eqnarray*}
&&C\,\int\limits_{x_1}^{\infty}
\Big(\frac{\varrho_1+|\zeta|}{\varrho_1+\varrho_2+|\zeta|}\Big)^{-a}
\Big(\frac {\varrho_2}{\zeta}\Big)^{c}
\,\exp\left(-\frac{|\zeta-y_1|^2}{\varrho_2^2}\right)\,\frac {d\zeta}{\varrho_2}\\
&&\le C\int\limits_{\frac{x_1}{\varrho_2}}^{\infty}\Big(\frac{\xi}{\xi+1}\Big)^{-a}\xi^{-c}
\,\exp\left(-\big|\xi-\frac {y_1}{\varrho_2}\big|^2\right)\,d\xi\\
&&\le C\,\Big(\Big(\frac{x_1}{\varrho_2}\Big)^{(1-c-\varepsilon)_-}
+\Big(\frac{x_1}{\varrho_2}\Big)^{(1-a-c-\varepsilon)_-}\Big),
\end{eqnarray*}
 which gives (\ref{Ap1b}) in the case 1. \medskip

{\it Case 2}: $x_1\le 2\Lambda\varrho_1$ is divided into two subcases.\medskip

{\it Subcase 2.1}: $\varrho_1\le \varrho_2$. Then we split the interval $(x_1,\infty)$ into two
parts. The integral over $(2\Lambda\varrho_1,\infty)$ is estimated
in the same way as in the case 1. Further, the integral over $(x_1,2\Lambda\varrho_1)$ is estimated by
\begin{eqnarray*}
&&C\,\Big(\frac{\varrho_1}{\varrho_1+\varrho_2}\Big)^{-a}
\int\limits_{x_1}^{2\Lambda\varrho_1}
\Big(\frac {\varrho_1}{\zeta-\phi(\varrho_1)}\Big)^{b}
\Big(\frac {\varrho_2}{\zeta-\phi(\varrho_1)}\Big)^{c}\,\frac {d\zeta}{\varrho_2}
\le C\,\Big(\frac{\varrho_1}{\varrho_1+\varrho_2}\Big)^{-a}
\\
&&
\times\!\!\!\int\limits_{\frac {x_1-\phi(\varrho_1)}{\varrho_2}}^{\frac {3\Lambda\varrho_1}{\varrho_2}}
\Big(\frac {\varrho_1}{\xi\varrho_2}\Big)^{b}\xi^{-c}\,d\xi
\le C\,\Big(\frac{\varrho_1}{\varrho_1+\varrho_2}\Big)^{-a}
\,\Big(\frac {\varrho_1}{\varrho_2}\Big)^{1-c}\,
\Big(\frac {x_1-\phi(\varrho_1)}{\varrho_1}\Big)^{(1-b-c-\varepsilon)_-},
\end{eqnarray*}
 which gives (\ref{Ap1b}) in the subcase 2.1. \medskip

{\it Subcase 2.2}: $\varrho_1\ge \varrho_2$. Then
$\varrho_1+|\zeta|\asymp\varrho_1+\varrho_2+|\zeta|$ for $\zeta\ge x_1$, and the integral is
estimated by
\begin{eqnarray*}
&&C\,\int\limits_{x_1}^{\infty}
\Big(\frac{\varrho_1+|\zeta|}{\zeta-\phi(\varrho_1)}\Big)^{b}
\Big(\frac {\varrho_2}{\zeta-\phi(\varrho_1)}\Big)^{c}
\,\exp\left(-\frac{|\zeta-y_1|^2}{\varrho_2^2}\right)\,\frac {d\zeta}{\varrho_2}\\
&&\le C\int\limits_{\frac{x_1-\phi(\varrho_1)}{\varrho_2}}^{\infty}
\Big(1+\frac{\varrho_1}{\xi\varrho_2}\Big)^{b}\xi^{-c}
\,\exp\left(-\big|\xi-\frac {y_1+\phi(\varrho_1)}{\varrho_2}\big|^2\right)\,d\xi
\le C\,\Big(\frac {\varrho_1}{\varrho_2}\Big)^{b}\\
&&\times\Big(\frac {x_1-\phi(\varrho_1)}{\varrho_2}\Big)^{(1-b-c-\varepsilon)_-}
\le C\,\Big(\frac {x_1-\phi(\varrho_1)}{\varrho_1}\Big)^{(1-b-c-\varepsilon)_-}
\Big(\frac{\varrho_1}{\varrho_2}\Big)^{\min\{b,1-c\}}\!,
\end{eqnarray*}
 which gives (\ref{Ap1b}) in the subcase 2.2.
\end{proof}

\begin{lem}\label{Lozenka1} Let $d\in\mathbb N$, $\varrho_1,\varrho_2>0$ and $a,\,b>-d$, $a+b>-d$.
Then
\begin{eqnarray}\label{Okt22a}
&&\int\limits_{\mathbb R^d}\exp\Big(-\frac{|x-z|^2}{\varrho_1}\Big)\exp\Big(-\frac{|z-y|^2}{\varrho_2}\Big)
\Big(\frac{|z|}{|z|+\sqrt{\varrho_1}}\Big )^{a}
\Big(\frac{|z|}{|z|+\sqrt{\varrho_2}}\Big )^{b}
d\zeta\nonumber\\
&&\leq C\,\frac{\varrho_1^{\frac{d+b_-}2}
\varrho_2^{\frac{d+a_-}2}}{(\varrho_1+\varrho_2)^{\frac{d+a_-+b_-}2}}\,\exp\Big(-\frac{|x-y|^2}{\varrho_1+\varrho_2}\Big),
\end{eqnarray}
where $C$ may depend on $a$ and $b$.
\end{lem}

\begin{proof}  (i) The case $a, b\geq 0$ is well known.\medskip

\noindent  (ii) Let $a, b< 0$. Without loss of generality we can assume that $\varrho_1\geq \varrho_2$.
Then the integral over the set $\{|z|>\sqrt{\varrho_2}\}$ in (\ref{Okt22a}) is estimated by
\begin{equation*}
\int\limits_{\mathbb R^d}\exp\Big(-\frac{|x-z|^2}{\varrho_1}\Big)\exp\Big(-\frac{|y-z|^2}{\varrho_2}\Big)
\Big(\frac{\sqrt{\varrho_2}}{\sqrt{\varrho_2}+\sqrt{\varrho_1}}\Big )^{a} dz.
\end{equation*}
The last expression does not exceed the right-hand side of (\ref{Okt22a}) by {\it (i)}.

Since
$$
|x-y|^2\leq\frac{\varrho_1+\varrho_2}{\varrho_1}\,|x-z|^2+\frac{\varrho_1+\varrho_2}{\varrho_2}\,|y-z|^2,
$$
the integral over the set $\{|z|<\sqrt{\varrho_2}\}$ is estimated by
\begin{eqnarray*}
&&C\exp\Big(-\frac{|x-y|^2}{\varrho_1+\varrho_2}\Big)\int\limits_{|z|<\sqrt{\varrho_2}}
\Big(\frac{|z|}{\sqrt{\varrho_2}+\sqrt{\varrho_1}}\Big )^{a}\Big(\frac{|z|}{\sqrt{\varrho_2}}\Big )^{b} dz\\
&&\le C\exp\Big(-\frac{|x-y|^2}{\varrho_1+\varrho_2}\Big)\frac{\varrho_2^{\frac {d+a}2}}{(\varrho_1+\varrho_2)^{\frac{a}2}}.
\end{eqnarray*}
The last expression also does not exceed the right-hand side of (\ref{Okt22a}).\medskip

\noindent (iii) The remaining cases can be easily reduced to the case {\it (ii)}.
\end{proof}

We formulate two auxiliary results on estimates of integral operators.
The first statement is proved in
\cite[Lemmas A.1 and A.3 and Remark A.2]{KN} for $\varepsilon_1=\varepsilon_2=0$,
see also \cite[Lemmas 2.1 and 2.2]{Na}. In general case the proof runs in the same way
by using \cite[Lemma A.3]{KN1}.

\begin{prop}\label{L_p}
Let $1\le p\le\infty$, $\sigma>0$, $0<r\le 2$,
$0\le\varepsilon_1<\frac 1p$, $0\le\varepsilon_2<1-\frac 1p$, ${\lambda_1}+{\lambda_2}>-m$,
and let
\begin{equation}\label{mu_m}
-\frac mp-\lambda_1<\mu<m-\frac mp+\lambda_2.
\end{equation}
Suppose also that the kernel ${\cal T}(x,y;t,s)$ satisfies the inequality
\begin{equation*}
|{\cal T}(x,y;t,s)|\le C\,\frac {{\cal R}_{x,t-s}^{\lambda_1+r}
{\cal R}_{y,t-s}^{\lambda_2}|x'|^{\mu-r}}
{(t-s)^{\frac {n+2-r}2}|y'|^{\mu}\,r_x^{\varepsilon_1}r_y^{\varepsilon_2}}\,\exp
\left(-\frac{\sigma|x-y|^2}{t-s} \right),
\end{equation*}
 for $t>s$. Then the integral operator ${\cal T}$ is bounded in
$L_p(\mathbb R^n\times\mathbb R)$ and in $\widetilde L_{p,\infty}(\mathbb R^n\times\mathbb R)$.
\end{prop}

The next proposition is \cite[Lemma A.4]{KN1}.

\begin{prop}\label{L_p_1}
Let $1<p<\infty$, $\sigma>0$, $\varkappa>0$, $0\le r\le 2$,
$0\le\varepsilon_1<\frac 1p$, $0\le\varepsilon_2<1-\frac 1p$,
${\lambda_1}+{\lambda_2}>-m$ and let $\mu$ be subject to {\rm (\ref{mu_m})}.
Also let the kernel ${\cal T}(x,y;t,s)$ satisfy the inequality
\begin{equation*}
|{\cal T}(x,y;t,s)|  \le C\,\frac {\delta^\varkappa\,{\cal
R}_{x}^{\lambda_1+r} {\cal R}_{y}^{\lambda_2}\,|x'|^{\mu-r}}
{(t-s)^{\frac{n+2-r}2+\varkappa}\,|y'|^{\mu}\,r_x^{\varepsilon_1}r_y^{\varepsilon_2}}
 \exp\left(-\frac{\sigma|x-y|^2}{t-s} \right),
\end{equation*}
for $x,y\in{\cal K}$ and $t>s+\delta$. Then for any $s^0>0$ the norm of the operator
$${\cal T}\ :\ L_{p,1}(\mathbb R^n\times\ (s^0-\delta,s^0+\delta))\ \to \
L_{p,1}(\mathbb R^n\times\ (s^0+2\delta,\infty))$$ does not exceed
a constant $C$ independent of $\delta$ and $s^0$.
\end{prop}

\bigskip

We thank S. V. Kislyakov for useful discussions.
V.~K. was supported by the Swedish Research Council (VR).
A.~N. was supported by RFBR grant 15-01-07650 and by St{.} Petersburg University grant 6.38.670.2013.
He also acknowledges the Link\"oping University for the financial support of his visit
in February 2015.

\vspace{\baselineskip}

\end{document}